\documentclass[12pt]{article}
\usepackage{mathrsfs}
\usepackage{amssymb}
\usepackage{amsmath}
\usepackage{mathtools}
\usepackage{amsthm}
\usepackage{latexsym}
\usepackage{imakeidx}
\makeindex
\usepackage{graphicx}
\usepackage{xcolor}
\usepackage{tikz}
\usepackage{adjustbox}
\usepackage{pgfplots}
\usepackage{cases}
\pgfplotsset{compat=1.8}
\usepackage{enumitem}
\usepackage{protosem}
\usepackage[linkcolor=blue,colorlinks=true,urlcolor=red,bookmarksopen=true]{hyperref}
\usepackage[margin=1in]{geometry}

 \numberwithin{equation}{section}

\allowdisplaybreaks[4]

\newtheorem{lem}{Lemma}[section]
\newtheorem{prop}[lem]{Proposition}
\newtheorem{defi}[lem]{Definition}
\newtheorem{rmk}[lem]{Remark}

\newtheorem{thm}[lem]{Theorem}
\newtheorem{cor}[lem]{Corollary}

\begin{document}

\title{\bf Mean Curvature Flow for Isoparametric Submanifolds in  Hyperbolic Spaces}
\author{Xiaobo Liu \thanks{Research was partially supported by NSFC grants 12341105 and 12426402.}, \,\,\,  Wanxu Yang }
\date{}

\maketitle

\begin{abstract}
   	Mean curvature flows of isoparametric submanifolds in Euclidean spaces and spheres have been studied by Liu and Terng in \cite{X.CT} and \cite{X.C}. In particular, it was proved that such flows always have ancient solutions. This is also true for mean curvature flows of isoparametric hypersurfaces in hyperbolic spaces by a result of Reis and Tenenblat in \cite{S.H.T}. In this paper, we study mean curvature flows of isoparametric submanifolds in  hyperbolic spaces with arbitrary codimension.  In particular, we will show that they always have ancient solutions and study their limiting behaviors.
\end{abstract}

\section{Introduction}

The {\it mean curvature flow} (abbreviated as MCF) of a Riemannian submanifold $M$ in a semi-Riemannian manifold X is a map
$f: M\times I \rightarrow X$ satisfying
  \begin{align}
  	\frac{\partial f}{\partial t}= H(\cdot,t), \label{1.16}
  \end{align}
where $I$ is an interval containing $0$, $f(\cdot, 0)$ is the immersion of $M$ in $X$, and $H(\cdot,t)$ is the mean curvature vector field of $f(\cdot,t)$ for all $t \in I$.
A solution to this equation is {\it ancient} if it exists over an interval $I = (-\infty, T)$ for some $T > 0$.
Ancient solutions play an important role in the study of singularities of MCF. A large class of ancient solutions
are given by MCF of isoparametric submanifolds.

A submanifold in a space form is {\it isoparametric} if its normal bundle is flat and principal curvatures along any parallel normal vector field are constant (cf. \cite{T.C} and \cite{B.W}).
In \cite{X.CT} and \cite{X.C}, Terng and the first author of this paper studied MCF
of isoparametric submanifolds $M$ in Euclidean spaces and spheres.
In particular, a recursive method to construct explicit solutions of MCF for such submanifolds were given in \cite{X.CT}.
It was also proved that MCF of $M$ always converges to a focal submanifold in a finite positive time if $M$ is not minimal.
In \cite{X.C}, it was proved  that MCF of all $M$ have {\it ancient solutions}. This gives many explicit examples with complex topological types for ancient solutions of MCF in Euclidean spaces and spheres. These examples were used to study sharpness of conditions in  Huisken and Sinestrari's rigidity theorem about ancient solutions for MCF of hypersurfaces in spheres modeled on shrinking
spherical caps (cf. \cite{HS}). Rigidity conjectures modeled on MCF of general isoparametric hypersurfaces in spheres
were also proposed in \cite{X.C}.
On the other hand, in \cite{S.H.T}, Reis and Tenenblat gave explicit ancient solutions for MCF of isoparametric
 hypersurfaces in hyperbolic spaces $\mathbb{H}^{m}(-1)$ with sectional curvature equal to $-1$.

 A notion of isoparametric submanifolds in general Riemannian manifolds was introduced in \cite{HLO}.
Equifocal submanifolds in compact symmetric spaces defined by Terng and Thorbergsson in \cite{TTh}
is equivalent to isoparametric submanifolds with flat sections in \cite{HLO}.
Results analogous to those in \cite{X.CT} were obtained for MCF of equifocal submanifolds in compact symmetric spaces by Koike in \cite{K.N}, and for MCF of regular leaves in generalized isoparametric foliations with compact leaves on compact Riemannian manifolds by Alexandrino and Radeschi in \cite{A.R}. The compactness condition for leaves and ambient spaces in \cite{A.R} was replaced by a weaker conditon for the compactness of leaf spaces by Alexandrino, Cavenaghi and Gon\c{c}alves in \cite{A.C.G}.
Results analogous to those in \cite{X.C} were obtained for MCF of regular leaves of polar foliations (they are in fact isoparametric) in simply connected symmetric spaces with non-negative curvature by Radeschi and the first author of this paper in \cite{L.R}.
In \cite{K.N1}, Koike also studied MCF for curvature-adapted isoparametric submanifolds with flat sections in non-compact
symmetric spaces. Note that isoparametric  submanifolds with flat sections in hyperbolic spaces must be hypersurfaces.
Above results do not cover MCF for isoparametric submanifolds in hyperbolic spaces with codimension bigger than one.

The purpose of this paper is to study MCF for isoparametric submanifolds in hyperbolic spaces with arbitrary codimension.
One of the  main results of this paper is the following

\begin{thm} \label{1.17}
  	Let $M$ be an $n$ dimensional complete  isoparametric submanifold in $\mathbb{H}^{m}(-1)$, and $f(\cdot,t)$ the  MCF of $M$ in $\mathbb{H}^{m}(-1)$. Then $f(\cdot,t)$ is always an ancient solution. Moreover, let $(-\infty,T)$ be the maximum existence interval for $f(\cdot,t)$ with
   $T \geq 0$. Then we have the following:
\begin{itemize}
\item[(1)] If $T< \infty$, then $f(\cdot,t)$ collapses to a focal submanifold of $M$  as $t\rightarrow T$.

\item[(2)] If $T= \infty$ and $M$ is not flat, then $f(\cdot,t)$ converges to a totally geodesic  submanifold of $\mathbb H^{m}(-1)$ as $t\rightarrow \infty$.

\item[(3)]
  If $T= \infty$ and $M$ is flat with dimension greater than $1$, then $f(\cdot,t)$ converges to a point on the ideal boundary of $\mathbb{H}^{m}(-1)$ as $t\rightarrow \infty$.

\item[(4)] Assume $M$ is not totally geodesic in $\mathbb{H}^m(-1)$ (otherwise $f(\cdot,t)$  is independent of $t$). As $t\rightarrow-\infty$, $f(\cdot,t)$ converges to a smooth submanifold  $\widetilde M$ with flat normal bundle in the ideal boundary of $\mathbb{H}^{m}(-1)$. Moreover, $\widetilde M$ and $M$ have the same dimension.
\end{itemize}
\end{thm}

\noindent
The precise meaning of smooth submanifold with flat normal bundle in the ideal boundary of $\mathbb{H}^{m}(-1)$ will be given in Definition \ref{def:SubmBD}. In this paper,
we will also prove that minimal isoparametric submanifolds in $\mathbb{H}^{m}(-1)$ are always totally geodesic (see Lemma \ref{4.2.1}).
Note that $M$ is always flat when its dimension is one. Hence, the flatness condition in Part (3) of Theorem \ref{1.17}
does not accurately reflect the type of $M$. We will discuss how to modify part (3) when dimension of $M$ is one
in Remark \ref{dimension,one}.

During the proof of Theorem \ref{1.17}, we also obtain a recursive method for constructing explicit solutions for MCF of
isoparametric submanifolds in hyperbolic spaces (see Remark \ref{rem:RecCons}). The relation between
MCF of such submanifolds in hyperbolic spaces and MCF of them considered as submanifolds in Lorentz spaces
plays a vital role in the proof of above theorem. However, unlike in Euclidean spaces and spheres considered in \cite{X.C}, MCF of some isoparametic submanifolds in Lorentz spaces do not have ancient solutions.

This paper is organized as follows:
In Section \ref{1.4}, we review basic definitions and properties for isoparametric submanifolds and ideal boundaries of
hyperbolic spaces. In Section \ref{1.10}, we study  MCF for isoparametric submanifolds in Lorentz spaces. The proof of Theorem~\ref{1.17} will be given in Section \ref{1.6}.

The authors would like to thank an anonymous referee for suggestions which is very helpful in improving the presentation of this paper.

\section{Preliminary} \label{1.4}

In this section, we review basic concepts and properties for isoparametric submanifolds and ideal boundaries
of hyperbolic spaces, which will be needed later in this paper. We will also clarify the meaning of
smooth submanifolds with flat normal bundles in the ideal boundaries of hyperbolic spaces.

\subsection{Isoparametric submanifolds in hyperbolic spaces.}

For any non-negative integer $m$, let $\mathbb R^{m,1}$ be the standard Lorentz space, which is an $m+1$ dimensional real vector space equipped with the following Lorentzian metric:
      \begin{align*}
       \left<x,y\right>=\sum_{i=1}^{m} x_i y_i -x_{m+1} y_{m+1}
      \end{align*}
for $ x=(x_1, \cdots, x_{m+1}),  y=(y_1,\cdots, y_{m+1}) \in \mathbb R^{m,1}$.
We refer to the book \cite{N.S} for general theory of Lorentzian geometries.
The hyperbolic space $\mathbb{H}^{m}( -r)$ with sectional curvature $-1/r$ can be realized as a submanifold
of $\mathbb R^{m,1}$ with the induced Riemannian metric in the following way
      \begin{align*}
      	\mathbb{H}^{m}(-r)=\{x\in \mathbb R^{m,1}\mid \left<x,x\right>=-r,  \, x_{m+1}>0 \},
      \end{align*}
where $r>0$ is a constant.
It is well known (cf. \cite{T.C}, \cite{T.C.2}, \cite{B.W}) that any isoparametric submanifold in  {\it space forms}  (i.e. Euclidean spaces, spheres, and Hyperbolic spaces) is an open part of a complete isoparametric submanifold. Without loss of generality, we require all isoparametric  submanifolds in space forms are complete throughout this paper.

The concept of isoparametric submanifolds in space forms can be extended to Lorentz spaces.	
 A submanifold $M$ in $\mathbb R^{m,1}$ is said to be {\it Riemannian} (also called spacelike in some literature) if the induced metric on $M$ is positive definite.
$M$ is called  an {\it isoparametric submanifold} in $\mathbb R^{m,1}$ if $M$ is a Riemannian submanifold, the normal bundle of $M$ in $\mathbb R^{m,1}$ is globally  flat, and the principal curvatures along any parallel normal vector field are constant (cf.  \cite{W.B}).
The following lemma is essentially due to Wu (cf. \cite{B.W}):
\begin{lem} \label{1.15}
 For any submanifold $M \subset \mathbb{H}^{m}(-r)\subset \mathbb R^{m,1} $,
	$M$ is isoparametric in  $\mathbb{H}^{m}(-r)$ if and only if $M$ is isoparametric in $\mathbb R^{m,1}$.
\end{lem}

\noindent
In fact, Proposition 1.10 in \cite{B.W} gives one direction of this lemma. The proof of another direction is straightforward.

A submanifold $M$ of $\mathbb{H}^{m}(-1)$  is {\it full} if $M$ is not included in any totally umbilical hypersurface of $\mathbb{H}^{m}(-1)$.
The following decomposition theorem was proved by Wu (see Corollary 2.6 in \cite{B.W}):
\begin{thm} \label{thm:Decomp}
	Assume  $M \subset \mathbb{H}^{m}(-1)\subset \mathbb R^{m,1}$. If $M$ is a full isoparametric submanifold,
then there exist $r \geq 1 $ and a  Lorentzian subspace $V$ of $\mathbb R^{m,1}$ with dimension $l+1$ such that
	\begin{align*}
		M=\mathbb{H}^{l}(- r)\times M',
	\end{align*}
	where $\mathbb{H}^l(- r) \subset V$ is one sheet of the hyperboloid
	$\{ x \in V \mid \left<x,x\right>=- r \}$
  and $M'$ is an isoparametric submanifold of the sphere $S^{m-l-1}(-1+r) \subset V^{\perp}$.
\end{thm}

\noindent
In this theorem $S^k(a)$ denotes the $k$-dimensional sphere with radius $\sqrt{a}$ centered at origin.

If an isoparametric submanifold $M \subset \mathbb{H}^{m}(-1)$ is not full, then it must be an isoparametric
submanifold of a totally umbilical hypersurface (with induced metric) in $\mathbb{H}^{m}(-1)$ (see Proposition 1.3 in \cite{B.W}).
It is well known that any totally umbilical complete hypersurface of $\mathbb{H}^{m}(-1)$ must have the form
$L(V,u)=\mathbb{H}^{m}(-1)\bigcap (V+u)$, where $V$ is a linear hyperplane in  $\mathbb R^{m,1}$ and $u \in \mathbb R^{m,1}$ (see, for example, \cite{ABD}, \cite{BCO} and \cite{B.W}). With the induced metric,
$L(V,u)$ is isometric to a Euclidean space, or a sphere, or a hyperbolic space if the restriction of
Lorentz metric $\left<\cdot, \, \cdot\right>$ to $V$ is  degenerate, or positive definite, or a nondegenerate symmetric bilinear form with index $1$ respectively. Hence if an isoparametric submanifold $M \subset \mathbb{H}^{m}(-1)$ is not full, it must be
an isoparametric submanifold in a Euclidean space, or a sphere, or a lower dimensional hyperbolic space.

\subsection{Ideal boundaries of hyperbolic spaces.}

Let $X$ be a complete Riemannian manifold with non-positive curvature.
Two geodesic rays $c_i: [0, \infty) \longrightarrow M$ for $i=1, 2$ are equivalent if the distance between $c_1(t)$ and $c_2(t)$ are bounded from above by a constant for all $t \geq 0$.
The {\it ideal boundary} of $X$, denoted by $\partial X$, is defined to be the set of all equivalence classes of geodesic rays in $X$.
There exists a natural topology on $\overline{X}=X \bigcup \partial X$ whose restriction to $X$ is the original topology
induced from the Riemannian metric on $X$
(see, for example, Chapter II.8 section 8.5 in \cite{B.H}).

To prove parts (3) and (4) of Theorem \ref{1.17}, we need a more explicit description of the ideal boundary of $\mathbb{H}^{m}(-r)$.
Let $\mathbb B^m$ be the open unit ball in $\mathbb R^m$ with the standard Euclidean metric, i.e.
 \begin{equation} \label{eqn:ball}
 \mathbb B^m :=\{y \in \mathbb R^m \mid \Vert y \Vert^2<1\}.
 \end{equation}
For $x=(x_1, \cdots, x_m, x_{m+1})\in \mathbb{H}^{m}(-1) \subset {\mathbb R}^{m,1}$, define
\begin{equation}  \label{9.2}
 \psi(x) :=  \left( \frac{x_1}{1+x_{m+1}}, \cdots, \frac{x_m}{1+x_{m+1}} \right).
\end{equation}
This gives a conformal diffeomorphism $\psi: \mathbb{H}^{m}(-1) \longrightarrow \mathbb B^m$.
Moreover, $\psi$ can be extended to a homeomorphism, also denoted by $\psi$, from $\overline{\mathbb{H}^{m}(-1)}$
to $\overline{\mathbb B^m} :=\{y \in \mathbb R^m \mid \Vert  y \Vert^2 \leq 1\}$.
In particular, it induces a homeomorphsim from the  ideal boundary of ${\mathbb{H}^{m}(-1)}$
to
\[ \partial \mathbb{B}^m := S^{m-1}=\{y \in \mathbb R^m \mid \Vert y \Vert^2 = 1\}.\]

\begin{rmk}
Note that $\psi^{-1}$ is the map given in Chapter I.6 section 6.7 of \cite{B.H}.
Pulling back the metric on ${\mathbb{H}^{m}(-1)}$ using $\psi^{-1}$, we obtain a hyperbolic
metric on ${\mathbb B^m}$ which has the form $4(1-\|y\|^2)^{-2} g_E$ for $y \in {\mathbb B^m}$, where $g_E$ is
the standard Euclidean metric.
${\mathbb B^m}$ with this hyperbolic metric is called the {\it Poincar\'{e} ball model}.
In this paper, we will always use the standard Euclidean metric on ${\mathbb B^m}$ unless otherwise stated.
\end{rmk}

More generally, let
$\epsilon = \{\epsilon_1,\cdots, \epsilon_{m+1}\}$ be an arbitrary
orthonormal basis of $\mathbb{R}^{m,1}$
satisfying the condition
\begin{equation} \label{eqn:condbasis}
\left<\epsilon_{i},\epsilon_{i}\right> = 1, \hspace{20pt}
\left<\epsilon_{m+1},\epsilon_{m+1}\right>=-1,  \hspace{20pt}
\left<\epsilon_{m+1},(0, \cdots, 0, 1) \right> < 0
\end{equation}
for $1 \leq i \leq m$.
For $a_1, \cdots, a_{m+1} \in \mathbb{R}$, define
\[ \Phi_\epsilon \left( \sum_{i=1}^{m+1} a_i \epsilon_i \right) := (a_1, \cdots, a_{m+1}). \]
Then $\Phi_\epsilon: \mathbb{R}^{m,1} \longrightarrow \mathbb{R}^{m,1}$
is an isometric transformation which preserves $\mathbb{H}^{m}(-r)$ for all $r > 0$.
We can define
 $\Psi_{\epsilon, r}: \mathbb{H}^{m}(-r) \longrightarrow \mathbb{B}^m $ by
 \begin{equation} \label{eqn:PsiEr}
 \Psi_{\epsilon, r} \left( x \right) := \frac{1}{a_{m+1} + \sqrt{r}} (a_1, \cdots, a_{m})
 \end{equation}
for $x=\sum_{i=1}^{m+1} a_i \epsilon_i \in \mathbb{H}^{m}(-r)$.
Then $\Psi_{\epsilon, r}$ is a conformal diffeomorphism which can be extended to a homeomorphism,
also denoted by $\Psi_{\epsilon, r}$,
from $\overline{\mathbb{H}^{m}(-r)}$ to $\overline{\mathbb{B}^m}$.
In particular, it induces a homeomorphism from ideal boundary $\partial {\mathbb{H}^{m}(-r)}$ to $\partial {\mathbb{B}^m}$.
This gives an identification of $\partial {\mathbb{H}^{m}(-r)}$ with $\partial {\mathbb{B}^m}$ which depends on
the choice of the orthonormal basis $\epsilon$.

Let $e = \{e_1, \cdots, e_{m+1}\}$ be the standard basis of $\mathbb{R}^{m,1}$, i.e.
$e_i = (0, \cdots, 0, 1, 0, \cdots, 0)$ where the $i$-th coordinate is $1$ and all other coordinates are $0$.
Then the map $\Psi_{e, r} \circ \Psi_{\epsilon, r}^{-1}$ is a homeomorphism from $\overline{\mathbb{B}^m}$ to itself, which is also a diffeomorphism when restricted to ${\mathbb{B}^m}$.
We have the following
\begin{lem} \label{lem:bdtrans}
Let $\Theta: S^{m-1} \longrightarrow S^{m-1}$ be the restriction of $\Psi_{e, r} \circ \Psi_{\epsilon, r}^{-1}$
to $\partial \mathbb{B}^m = S^{m-1}$. Then $\Theta$ is a conformal transformation with respect to the standard
metric on the unit sphere $S^{m-1}$.
\end{lem}

\noindent
{\bf Proof}:
For each $i=1, \cdots, m+1$, write $\epsilon_i=(v_i, c_i)$ where $v_i \in \mathbb{R}^m$ consists of the first $m$ coordinates of $\epsilon_i$ and $c_i \in \mathbb{R}$ is the last coordinate of $\epsilon_i$. Then for every $a=(a_1, \cdots, a_m) \in {\mathbb{B}^m}$,
\[ \Psi_{e, r} \circ \Psi_{\epsilon, r}^{-1} (a)
=  \frac{(1+\|a\|^2) v_{m+1} +2 \sum_{i=1}^m a_i v_i }{1-\|a\|^2+(1+\|a\|^2) c_{m+1} + 2 \sum_{i=1}^m a_i c_i}.\]
Therefore, if $\|a\|=1$, i.e. $a \in S^{m-1} = \partial {\mathbb{B}^m}$, we have
\begin{equation}
 \Theta (a)
    = \frac{1}{h(a,c)} \left( v_{m+1} + \sum_{i=1}^m a_i v_i \right),
\end{equation}
where
\[ h(a,c) := c_{m+1} +  \sum_{i=1}^m a_i c_i.\]

To show $\Theta$ is a smooth map, we only need to show $h(a,c) \neq 0$ for all $a \in S^{m-1} $.
We first observe that $c_i = - \langle\epsilon_i, e_{m+1}\rangle $ for all $i$. Hence
\[  e_{m+1} = c_{m+1} \epsilon_{m+1} - \sum_{i=1}^m c_i \epsilon_i, \]
and
\[ - c_{m+1}^2 + \sum_{i=1}^m c_i^2 \, = \, \langle e_{m+1}, e_{m+1}\rangle = -1. \]
Together with condition \eqref{eqn:condbasis}, we obtain $c_{m+1} \geq 1$. So
by Cauchy inequality, we have
\[ h(a,c) \geq c_{m+1} - \|a\| \left( \sum_{i=1}^m c_i^2 \right)^{\frac{1}{2}}= c_{m+1} - \sqrt{ c_{m+1}^2 -1} > 0 \]
 for all $a \in S^{m-1}$.
Hence $\Theta$ is smooth. Since we already know $\Theta$ is a homeomorphism by construction, this also shows that
$\Theta$ is a diffeomorphism from $S^{m-1}$ to itself.

Now we prove that $\Theta$  is a conformal map. In fact, for
$w=(w_1, \cdots, w_m) \in T_a S^{m-1} \subset \mathbb{R}^m,$
 we have
\[ \Theta_*(w) = \frac{1}{h(a,c)^2} \left( s(c,w) v_{m+1} + \sum_{i=1}^m  r_i(a, c ,w) v_i \right), \]
where
\[ s(c, w):= - \sum_{i=1}^m c_i w_i, \hspace{20pt} r_i(a, c ,w):= h(a,c) w_i + s(c, w) a_i.\]
Since $\epsilon_1, \cdots, \epsilon_{m+1}$ are orthonormal in $\mathbb{R}^{m,1}$, we have
\[ \|v_{m+1}\|^2 = c_{m+1}^2-1, \hspace{10pt}
\|v_i\|^2 = c_i^2+1, \hspace{10pt}
\langle v_j, v_k\rangle = c_j c_k \]
for all $i = 1, \cdots, m$ and $j, k = 1, \cdots, m+1$ with $j \neq k$.
Hence
\begin{eqnarray*}
 h(a,c)^4 \| \Theta_*(w) \|^2
 &=& \left( s(c,w) c_{m+1} + \sum_{i=1}^m r_i(a,c,w) c_i \right)^2
    - s(c, w)^2 + \sum_{i=1}^m r_i(a,c,w)^2
\end{eqnarray*}
It is straightforward to check that
\[ s(c,w) c_{m+1} + \sum_{i=1}^m r_i(a,c,w) c_i = 0\]
and
\[ - s(c, w)^2 + \sum_{i=1}^m r_i(a,c,w)^2 = h(a,c)^2 \|w\|^2\]
using the fact $\|a\|^2=1$ and $\sum_{i=1}^m a_i w_i = 0$ since $w \in T_a S^{m-1}$.
Therefore we have proved that
\begin{equation}
\|\Theta_*(w) \|^2 = \frac{\|w\|^2}{\left( c_{m+1}+\sum_{i=1}^m a_i c_i \right)^2}
\end{equation}
for all $a \in S^{m-1}$ and $w \in T_a S^{m-1}$.
This shows that $\Theta$ is a conformal map and the lemma is thus proved.
$\Box$

\begin{defi} \label{def:SubmBD}
Let $M$ be a subset of $\partial \mathbb{H}^m(-r)$. $M$ is called a smooth submanifold in the ideal boundary
of $\mathbb{H}^m(-r)$ if $\Psi_{\epsilon, r}(M)$ is a smooth submanifold of $S^{m-1}$ for any
orthonormal basis $\epsilon$ of $\mathbb{R}^{m,1}$ satisfying condition \eqref{eqn:condbasis}.
We say $M$ has flat normal bundle in the ideal boundary if
the normal bundle of  $\Psi_{\epsilon, r}(M)$ in $S^{m-1}$ is flat with respect to the standard metric on $S^{m-1}$.
\end{defi}
Lemma \ref{lem:bdtrans} shows that the notion of "smooth submanifold in the ideal boundary" is independent
of the choice of orthonormal basis $\epsilon$. Together with the lemma below, it also implies that
the notion for
"flatness of normal bundle of $M$"  does not depend on the choice of $\epsilon$.
Hence both notions are well defined.

\begin{lem}\label{flat}
Let $g'$ and $g$ be two conformally equivalent Riemannian metrics on a smooth manifold $N$
such that $g' =e^{2\rho}g$, where $\rho$ is a smooth function on $N$.
Assume $M$ is an $n$-dimensional submanifold of $N$.
\begin{itemize}
\item[(1)] Let  $R'^{\perp}$ and $R^{\perp}$ be the curvature operators on the normal bundle $\nu M$ of $M$ in $N$
with respect to $g'$ and $g$ respectively. Then we have
\[ R'^{\perp}=R^{\perp}.\]
In particular, $\nu M$ is flat with respect to $g'$ if and only if
it is flat with respect to $g$.

\item[(2)]  Let $H'(x)$ and $H(x)$ be the mean curvature vectors of $M$ at point $x$ with respect to $g'$ and $g$ respectively.
Then for all $x\in M$,
\[ H'(x)=-ne^{-2\rho(x)}({\rm grad} (\rho) \mid_x)^{\perp}+e^{-2\rho(x)}H(x), \]
 where superscript $\perp$ denotes the projection from $T_x N$ to the normal space $\nu_x M$,
 ${\rm grad}(\rho)$ is the gradient of $\rho$ with respect to $g$.
\end{itemize}
\end{lem}
\noindent
{\bf Proof}:
First observe that since $g'$ and $g$ are conformally equivalent, the normal bundles of $M$ with respect to these two metrics are the same, which is denoted by $\nu M$. In particular, the normal projection $\perp$ is well defined.
Let  $\overline D$, $D$, $D^{\perp}$  (or $\overline{D'}$, $D'$, $D'^{\perp}$) be the Levi-Civita connections on $TN$, $TM$, $\nu M$ with respect to $g$ (or $g'$) respectively. For any local vector fields $X$ and $Y$ on $N$, it is well known that
\begin{equation} \label{eqn:connconf}
\overline{D'}_X Y = \overline{D}_X Y + S(X,Y),
\end{equation}
where
\begin{equation} \label{eqn:S}
 S(X,Y) := X(\rho) \, Y+Y(\rho) \, X - g(X,Y) \, {\rm grad}(\rho)
\end{equation}
(see, for example, Chapter 8 Exercises 5 in \cite{d.C}).

\vspace{6pt}
\noindent
{Proof of part (1)}:
For any local tangent vector field  $Y$ on $M$ and local normal vector field $\eta$ on $M$, we have
\begin{align} \label{eqn:STN}
	S(Y,\eta)^{\perp}&= \big( Y(\rho) \, \eta + \eta(\rho) \, Y - g(Y,\eta) \, {\rm grad}(\rho) \big)^{\perp} \,
	 =Y(\rho) \, \eta,
\end{align}
and
\begin{align} \label{eqn:DTN}
	D'^{\perp}_{Y} \, \eta
	&=\big( \overline{D'}_{Y} \, \eta \big)^{\perp}
	\, = \, \big( \overline{D}_{Y} \, \eta + S(Y,\eta) \big)^{\perp}
	\, = \, D^{\perp}_{Y} \, \eta + Y(\rho) \, \eta.
\end{align}
By equation \eqref{eqn:DTN}, for any local tangent vector fields $X$, $Y$ on $M$ and local normal vector field $\eta$
on $M$, we have
\begin{align}
	D'^{\perp}_{X} D'^{\perp}_{Y} \, \eta
	&=D^{\perp}_X \big( D'^{\perp}_{Y} \, \eta \big)
           + X(\rho)  D'^{\perp}_{Y} \, \eta \nonumber \\
	&=D^{\perp}_X \big( D^{\perp}_Y \, \eta + Y(\rho) \, \eta \big)
           + X(\rho) \big( D^{\perp}_Y \, \eta + Y(\rho) \, \eta \big)\nonumber \\
	&=D^{\perp}_X D^{\perp}_Y \, \eta +
             X Y(\rho) \, \eta + Y(\rho) \, D^{\perp}_X \, \eta
             + X(\rho) D^{\perp}_Y \, \eta + X(\rho) Y(\rho) \, \eta. \nonumber
\end{align}
Hence
\begin{align}
	D'^{\perp}_{X} D'^{\perp}_{Y} \, \eta - D'^{\perp}_{Y} D'^{\perp}_{X} \, \eta
	&=D^{\perp}_X D^{\perp}_Y \, \eta - D^{\perp}_Y D^{\perp}_X \, \eta
             + [X, Y](\rho) \, \eta  . \label{11.2}
\end{align}
Using equations \eqref{eqn:DTN} and \eqref{11.2}, we have
\begin{align*}
	R'^{\perp}(X,Y) \eta
    =& \big( D'^{\perp}_{X} D'^{\perp}_{Y} \, \eta - D'^{\perp}_{Y} D'^{\perp}_{X} \, \eta \big)
            - D'^{\perp}_{[X,Y]} \, \eta \\
	=& \big( D^{\perp}_X D^{\perp}_Y \, \eta - D^{\perp}_Y D^{\perp}_X \, \eta
             + [X, Y](\rho) \, \eta  \big)
	    -\big( D^{\perp}_{[X,Y]} \, \eta + [X,Y](\rho) \eta \big)\\
	=& R^{\perp}(X,Y)\eta.
\end{align*}
By definition, $\nu M$ is flat if its curvature operator is $0$. This finishes the proof of part (1)
of this lemma.

\vspace{6pt}
\noindent
{Proof of part (2)}:
Let $A'$ and $A$ be the shape operators  of $M$ in $N$  with respect to metrics $g'$ and $g$ respectively.
By equations \eqref{eqn:connconf} and \eqref{eqn:S}, for any local tangent vector field $X$ on $M$ and local normal
vector field $\eta$ on $M$, we have
\begin{equation} \label{eqn:A'}
A'_{\eta} X
=  - (\overline{D'}_{X} \eta)^{\top}
= - \left( \overline{D}_{X} \eta + S(X, \eta) \right)^{\top}
= A_{\eta} X - \eta(\rho) X,
\end{equation}
where superscript $\top$ means projection from $TN$ to $TM$.

Let $k$ be the codimension of $M$ in $N$.
Choose a local orthonormal frame $\{e_i \mid i=1, \cdots, n\}$ for $TM$ and
a local orthonormal  frame $\{\eta_\alpha \mid \alpha=1, \cdots, k\}$ for $\nu M$
 with respect to metric $g$.
Then $\{e'_i := e^{-\rho} e_i \mid i=1, \cdots, n\}$ is a local orthonormal frame of $TM$
and $\{\eta'_\alpha := e^{-\rho}\eta_\alpha \mid \alpha=1, \cdots, k\}$ is a local  orthonormal  frame  for $\nu M$
with respect to metric $g' = e^{2\rho}g $.

By equation \eqref{eqn:A'}, for any $x\in M$, we have
\begin{align}
	H'(x)
	&=\sum_{i=1}^n \sum_{\alpha=1}^{k} g'( A'_{\eta'_{\alpha}} e'_i, \,  e'_i) \, \eta'_{\alpha}
    = e^{-2 \rho} \sum_{i=1}^n  \sum_{\alpha=1}^{k} g( A'_{\eta_{\alpha}} e_i, \,  e_i) \, \eta_{\alpha} \nonumber \\
    &= e^{-2 \rho} \sum_{i=1}^n  \sum_{\alpha=1}^{k} g( A_{\eta_{\alpha}} e_i - \eta_{\alpha} (\rho) e_i, \,  e_i)
          \, \eta_{\alpha}
    = e^{-2 \rho} \left( H(x) - n  ({\rm grad}(\rho))^\perp  \right).
     \nonumber
\end{align}
This finishes the proof of the lemma.
$\Box$

\begin{rmk}
Part (2) of Lemma \ref{flat} shows that the constancy of $\|H(x)\|$ may not be preserved under conformal transformations.
This indicates that we may not have a definition of minimal or isoparametric submanifolds in the ideal boundary of
 hyperbolic spaces which does not rely on
the choice of orthonormal basis $\epsilon$ of $\mathbb{R}^{m,1}$.
\end{rmk}

The following lemma will be useful in studying the limit of a curve in a hyperbolic space which converges to a point in the ideal boundary.

\begin{lem} \label{lem:convBD}
Given a curve $\gamma(t) \in \mathbb{H}^{m}(-r)$.
Let $\epsilon=\{\epsilon_1, \cdots, \epsilon_{m+1} \}$
be an orthonormal basis of $\mathbb{R}^{m,1}$ satisfying condition \eqref{eqn:condbasis}.
Write $\gamma(t) = \sum_{i=1}^{m+1} a_i(t) \epsilon_i$. Then $\gamma(t)$ converges to a point
in the ideal boundary $\partial \mathbb{H}^{m}(-r)$ as $t \rightarrow T$ for some $T$ if and only if
$ a_{m+1}(t) \rightarrow \infty$ and $\frac{\gamma(t)}{a_{m+1}(t)}$ converges to a point in
$\mathbb{R}^{m,1}$ as $t \rightarrow T$. In this case,
\begin{equation} \label{eqn:limitbd}
 \lim_{t \rightarrow T} \,\, \Psi_{\epsilon, r}(\gamma(t))
       \,\, = \,\,\lim_{t \rightarrow T} \,\, \frac{1}{a_{m+1}(t)} (a_1(t), \cdots, a_m(t)) \in S^{m-1}.
\end{equation}
\end{lem}
\noindent
{\bf Proof}:
First observe that $\sum_{i=1}^m a_i(t)^2 = a_{m+1}(t)^2 -r$ since $\gamma(t) \in \mathbb{H}^{m}(-r)$. Hence
\[ \|\Psi_{\epsilon, r}(\gamma(t))\|^2 = \frac{a_{m+1}(t) - \sqrt{r}}{a_{m+1}(t) + \sqrt{r}}, \]
and $\|\Psi_{\epsilon, r}(\gamma(t))\| \rightarrow 1$ if and only if
$ a_{m+1}(t) \rightarrow \infty$.

Since $\Psi_{\epsilon, r}: \overline{\mathbb{H}^{m}(-r)} \longrightarrow \overline{\mathbb{B}^m}$
is a homeomorphism, $\gamma(t)$ converges to a point
on $\partial \mathbb{H}^{m}(-r)$  if and only if $\Psi_{\epsilon, r}(\gamma(t))$ converges to a point
on $S^{m-1}$. This implies $ a_{m+1}(t) \rightarrow \infty$ and
equation \eqref{eqn:limitbd} holds since
$\lim_{t \rightarrow T} \,\, \frac{a_{m+1}(t)}{a_{m+1}(t)+\sqrt{r}} = 1$.
Moreover, existence for the limit on the right hand side of equation \eqref{eqn:limitbd}
is equivalent to the existence of
\[ \lim_{t \rightarrow T} \frac{\gamma(t)}{a_{m+1}(t)} = \sum_{i=1}^{m+1} \lim_{t \rightarrow T} \frac{a_i(t)}{a_{m+1}(t)} \epsilon_i,\]
which is a point in $\mathbb{R}^{m,1}$ if exists.
The lemma is thus proved.
$\Box$

\section{MCF for isoparametric submanifolds in Lorentz spaces} \label{1.10}\

Note that in Theorem \ref{thm:Decomp}, the ambient Lorentz space $\mathbb R^{m,1}$ splits as a product
of $V$ and $V^{\perp}$, but the hyperbolic space $\mathbb{H}^{m}(-1)$ does not split as a product of two Riemannian
manifolds. Moreover, $\mathbb R^{m,1}$ is a linear space while $\mathbb{H}^{m}(-1)$ is not linear.
Therefore it is convenient to study MCF in $\mathbb R^{m,1}$ instead of $\mathbb{H}^{m}(-1)$.
For a submanifold $M \subset \mathbb{H}^{m}(-1) \subset  \mathbb R^{m,1}$, we call MCF of $M$ in $\mathbb{H}^{m}(-1)$ (respectively
in $\mathbb{R}^{m,1}$) the {\it hyperbolic MCF} (respectively {\it Lorentzian MCF}) of $M$.
In this  section, we will study relations between hyperbolic MCF and Lorentzian MCF as well as properties of Lorentzian
MCF for isoparametric submanifolds.

We first recall a basic fact in semi-Riemannian geometry:
Assume $X$ is a semi-Riemannian  manifold and  $M, N$  are Riemannian submanifolds in $X$
with $M \subset N$.
Then the mean curvature vector field of $M$ as a submanifold in $N$ is just the orthogonal projection
of the mean curvature vector field of $M$ as a submanifold in $X$ to the tangent spaces of $N$.
In particular, we have
\begin{lem} \label{4.1}
		Assume $M$ is an $n$-dimensional submanifold of $\mathbb{H}^{m}(-r) \subset \mathbb R^{m,1}$. Let $H(x)$ and  $H^L(x)$  be the mean curvature vectors of $M$ in $\mathbb{H}^{m}(-r)$ and  $\mathbb R^{m,1}$ respectively at $x \in M$. Then
	\begin{align*}
		H^L(x) = H(x) + \frac{n}{r}x.
	\end{align*}
\end{lem}
\noindent
{\bf Proof}:
Let $T_x \mathbb{H}^{m}(-r)$ and $\nu_x \mathbb{H}^{m}(-r)$ be the tangent and normal spaces of $\mathbb{H}^{m}(-r)$ at $x$ respectively.
Since the orthogonal projection of $H^L(x)$ to $T_x \mathbb{H}^{m}(-r)$ is $H(x)$, we only need to show that the orthogonal
projection of  $H^L(x)$ to $\nu_x \mathbb{H}^{m}(-r)$  is $\frac{n}{r}x$.

Note that $\nu_x \mathbb{H}^{m}(-r)$ is spanned by
$\eta(x):=\frac{x}{\sqrt{r}}$.
Let $\nabla^{L} $ be the Levi-Civita connection on  $\mathbb R^{m,1}$.
Then $\nabla^{L}_v \eta = \frac{v}{\sqrt{r}}$ for any tangent vector $v$.
Let $\{e_1,\cdots,e_n\}$ be an orthonormal basis of $T_x M$ and $A^{L}$ the shape operator
of $M$ as a submanifold in $\mathbb R^{m,1}$. Then
\begin{align*}
		 	\left<H^{L}(x), \eta(x)\right>
		 	      &=\sum_{i=1}^n \left<A^{L}_{\eta(x)}e_i,e_i\right>
		 	       =-\sum_{i=1}^n\left<\nabla^L_{e_i}\eta, e_i\right>
		 	       =-\sum_{i=1}^n \frac{1}{\sqrt{r}} \left<e_i,e_i\right>
		 	       =-\frac{n}{\sqrt{r}}.
\end{align*}
Since $\left< \eta(x), \eta(x) \right> = -1$, the orthogonal projection of $H^{L}(x)$ to
$\nu_x \mathbb{H}^{m}(-r)$ is
		 \begin{align*}
			- \left<H^{L}(x), \eta(x)\right>\eta(x)
			    &= \frac{n}{r}x.
		\end{align*}
The lemma is thus proved. $\Box$

The relation between hyperbolic MCF and Lorentzian MCF is given by the following:
\begin{lem} \label{1.2}
		Let  $f(\cdot,t)$ and $F(\cdot,t)$ be respectively the hyperbolic MCF and Lorentzian MCF of an $n$-dimensional submanifold
$M$ in $\mathbb{H}^{m}(-r) \subset \mathbb R^{m,1}$.  Then for any $x \in M$,
		\begin{align}
			F(x,t)=\sqrt{1+\frac{2nt}{r}}f(x,\frac{r}{2n}\ln(1+\frac{2nt}{r})) \label{7.1.4}
		\end{align}
if both sides of this equation are well defined.
\end{lem}
\noindent
{\bf Proof}:
Let $s(t)=\frac{r}{2n}\ln(1+\frac{2nt}{r})$. The function on the right hand side of equation \eqref{7.1.4}
can be written as
		\begin{align*}
			G(x,t)=\sqrt{1+\frac{2nt}{r}}f(x,s(t)). 
		\end{align*}
Let $H(\cdot,s)$ be the mean curvature vector field of $f(\cdot,s)$ in $\mathbb{H}^{m}(-r)$. Then
$\frac{\partial}{\partial s} f(x,s) = H(x, s)$ since $f(x, s)$ satisfies the hyperbolic MCF equation.
By the chain rule, we have
\begin{align*}
			\frac{\partial}{\partial t} G(x,t)
			&=\frac{1}{\sqrt{1+\frac{2nt}{r}}} \left\{ H(x, s(t)) +\frac{n}{r}f(x,s(t)) \right\},
\end{align*}
which is exactly the mean curvature vector of $G(\cdot , t)$ as a submanifold in $\mathbb R^{m,1} $  at point $G(x, t)$ by Lemma \ref{4.1}. Moreover, $G(x, 0) = f(x,0) = F(x, 0)$. By the uniqueness of the solution to MCF equation,
we have $G(x, t) = F(x, t)$. This finishes the proof of the lemma.
$\Box$
	
An immediate consequence of this lemma is the following
\begin{cor} \label{cor:fF}
If  $f(.,t)$ exists over an interval  $(-\infty, T)$ for some $T\geq 0$, then $F(.,t)$ exists for all
$t\in (-\frac{r}{2n},\frac{r}{2n}(e^{\frac{2nT}{r}}-1))$.
On the other hand, if $F(.,t)$ exists over an interval $(-\frac{r}{2n},T')$ for some $T'\geq 0$, then
$f(.,t)$ exists for $t\in (-\infty, \frac{r}{2n}\ln(1+\frac{2nT'}{r}))$ and it is given by
\begin{align}
		f(x,t)=e^{-\frac{nt}{r}}F(x,\frac{r}{2n}(e^{\frac{2nt}{r}}-1)). \label{4.18}
\end{align}
\end{cor}

\begin{rmk} \label{rem:fF}
Since the image of $f(\cdot, t)$ always lies in the hyperbolic space $\mathbb{H}^{m}(- r)$,
for equations \eqref{7.1.4} and \eqref{4.18} to hold, the image of $F(\cdot, t)$ must also lie in a
hyperbolic space $\mathbb{H}^{m}(d)$ for some $d<0$. In particular, it must lie in the time cone
\begin{align} \label{eqn:timecone}
			\mathcal{TC} :=\{v\in  \mathbb R^{m,1}\mid \left<v,v\right><0\}.
\end{align}
Moreover, it is possible that $F(\cdot,t)$ may exist for $t\leq -\frac{r}{2n}$. In this case,
there is no hyperbolic MCF corresponds to $F(\cdot,t)$ for $t\leq -\frac{r}{2n}$.
For example,
let $P=\{ (x_1, x_2, 2) \mid x_1, x_2 \in \mathbb{R} \}  \subset \mathbb{R}^{2,1}$ and $M=P \bigcap \mathbb{H}^2(-1)$.
Then $P$ is a plane with induced Euclidean metric and $M$ is a circle in $P$.
The Lorentzian MCF $F(\cdot, t)$ of $M$ exists for all $t \in (-\infty, T)$ for some $T>0$, and it gives
a foliation of $P \setminus \{(0,0,2)\}$ by circles centered at $(0,0,2)$.
But only those circles inside the time cone $\mathcal{TC}$ correspond to
hyperbolic MCF $f(\cdot, t)$ through equation \eqref{7.1.4}. The circles outside the time cone
do not correspond to $f(\cdot, t)$ for any $t$. Note that for this example,
$f(\cdot, t)$ also exists for all $t \in (-\infty, T')$ for some $T'>0$
and the image of $f(\cdot, t)$ is also a circle  for each $t$
(possibly lies in different plane).
\end{rmk}

Since minimal submanifolds do not move under MCF, another immediate consequence of Lemma \ref{1.2} is the following:	
\begin{cor} \label{2.6}
		If $M$ is an $n$-dimensional minimal submanifold in $\mathbb{H}^{m}(-r)$ with inclusion map $f_0$,
then MCF of $M$ in $\mathbb{R}^{m,1}$ is given by
		\begin{align*}
			F(x,t)=\sqrt{1+\frac{2nt}{r}}f_0(x), \,\,\, {\rm for} \,\,\, t \in (- \frac{r}{2n}, \, \infty), x \in M.
		\end{align*}
\end{cor}

\begin{rmk} \label{1.11}
   As observed in \cite{X.C}, the Euclidean MCF of minimal submanifolds in the sphere always have ancient solutions.
   But for minimal submanifolds in hyperbolic spaces, their Lorentzian MCF may not have ancient solutions, as demonstrated in the above corollary.
\end{rmk}

Lemma \ref{1.2} allows us to study hyperbolic MCF via the study of Lorentzian MCF.
In particular, we can use this relation to obtain the following
\begin{prop}
Let $M$ be an isoparametric submanifold in $\mathbb{H}^{m}(-r)\subset \mathbb R^{m,1}$, $F(\cdot, t)$ and $f(\cdot, t)$
the Lorentzian and hyperbolic MCF of $M$ respectively.
Then $F(\cdot, t)$ and $f(\cdot, t)$ are isoparametric whenever they exist.
\end{prop}
\noindent
{\bf Proof}:
This proposition is the Lorentzian analog of Proposition 2.1 in \cite{X.CT}.
If $M$ is an isoparametric submanifold in a hyperbolic space $\mathbb{H}^m(-r)$, it is also isoparametric
in Lorentz space $\mathbb{R}^{m,1}$ by Lemma \ref{1.15}. As proved in \cite{B.W}, such submanifolds have
the structure of curvature distributions and curvature normals similar to those of isoparametric submanifolds
in Euclidean spaces obtained in \cite{T.C}. To show that $F(\cdot, t)$ is an isoparametric submanifold parallel
to $M$, we can use
the same proof as that of Proposition 2.1 in \cite{X.CT} with possibly a small modification for the sign due to the fact that $\left< x, x \right> < 0$ for $x \in \mathbb{H}^m(-r)$.
By Lemma \ref{1.2}, $f(\cdot, t)$ is also isoparametric.
$\Box$

\begin{rmk}	
The proof of Theorem 2.2 in \cite{X.CT} also works for Lorentzian MCF $F(\cdot, t)$ of an $n$-dimensional isoparametric
submanifold $M$ in a hyperbolic space. Hence for any $x \in M$, we have
	\begin{align}
				\left<F(x,t), F(x, t) \right> = \left<x, x \right> - 2nt \label{3.2}
	\end{align}
whenever $F(x, t)$ exists. This implies that for $F(x, t)$ staying in the time cone $\mathcal{TC}$, $t$ can not be smaller than $\frac{\left<x, x \right>}{2n}$. If $t \leq \frac{\left<x, x \right>}{2n}$, even if $F(\cdot, t)$ exists, it does not correspond to a hyperbolic MCF via equation \eqref{7.1.4}.
\end{rmk}

\section{Proof of Theorem \ref{1.17}} \label{1.6}	

In this section, we give a proof for Theorem \ref{1.17} based on Theorem \ref{thm:Decomp} and results for MCF
of isoparametric submanifolds in Euclidean spaces and spheres obtained in \cite{X.CT} and \cite{X.C}.
During the proof, we also give a recursive
construction for explicit solutions of MCF of isoparametric submanifolds in hyperbolic spaces.

We first consider the case when $M$ is full in $ \mathbb{H}^{m}(-1)$, i.e. $M$ is not contained in any umbilical hypersurface
in $\mathbb{H}^{m}(-1)$.
\begin{prop} \label{1.9}
		Let  $M$ be a full isoparametric submanifold in $\mathbb{H}^{m}(-1)$
and $f(\cdot,t)$ the solution to hyperbolic MCF of $M$.
  Then  $f(\cdot,t)$ is an ancient solution.
  Moreover, if $M \neq \mathbb{H}^{m}(-1)$,
   then there exists a finite $T > 0$ such that the maximal existence interval for $f(\cdot,t)$ is $(-\infty, T)$, and $f(\cdot,t)$ converges to a focal submanifold  of $M$ as $t\rightarrow T$.
   Moreover, as $t\rightarrow -\infty$, $f(\cdot,t)$ converges to a smooth submanifold $\widetilde M$ with flat normal bundle in the ideal boundary of $\mathbb{H}^{m}(-1)$ and $\widetilde M$ is diffeomorphic to $M$.
\end{prop}
\noindent
{\bf Proof}:
		Since $M$ is full and isoparametric, by Theorem \ref{thm:Decomp}, there exists $r \geq 1$ and a Lorentzian subspace $V$ of $\mathbb R^{m,1}$ with dimension $l+1$, such that
\begin{align*}
			M=\mathbb{H}^l(-r) \times M',
\end{align*}
where $\mathbb{H}^l(-r)$ is one sheet of the hyperboloid
$\{x \in V \mid \left<x, x\right>=-r\} \subset V$ and
$M'$ is an isoparametric submanifold in the sphere $S^{m-l-1}(r-1)\subset V^{\perp} $.
		
Let $F_1(\cdot,t),F_2(\cdot,t)$ be the solutions to MCF of $\mathbb{H}^l(-r)$ in $V$ and MCF of $M'$ in  $V^{\perp}$ respectively.  Since $\mathbb{R}^{m,1}$ splits as the product $V \times V^\perp$, the MCF of $M$ in $\mathbb{R}^{m,1}$
is given by
		\begin{align}
			F((x,y),t)=(F_1(x,t),F_2(y,t)) \label{7.1.6}
		\end{align}
for $x \in \mathbb{H}^l(-r)$ and $y \in M'$.

By Corollary \ref{2.6}, $F_1(\cdot,t)$ exists for all
$t\in (-\frac{r}{2l}, \, \infty)$. Since $r \geq 1$ and $l \leq n$ where $n$ is the dimension of $M$,
we have  $-\frac{r}{2l} \leq -\frac{1}{2n}$. Hence $F_1(\cdot,t)$ exists for all $t\in (-\frac{1}{2n}, \, \infty)$.

If $M'$ is just a point, then $M$ is contained in a translation of $V$.
Since $M$ is full, we must have $V=\mathbb{R}^{m,1}$ and $M=\mathbb{H}^m(-r)$.
In this case $f(\cdot, t)$ exists for all $t \in \mathbb{R}$.

If $M'$ is not a point, then by Theorem 1.1 in \cite{X.CT} and Theorem 1.1 in \cite{X.C},  $F_2(\cdot,t)$ exists  for all $t\in (-\infty, T')$ for some finite $T' > 0$ and collapses to a focal submanifold of $M'$ when $t$ goes to $T'$.
Thus, $F(.,t)$ exists for $t\in (-\frac{1}{2n},T')$, and it collapses to a focal submanifold  when $t$ goes to $T'$. Let $T = \frac{1}{2n} \ln(1+2n T')$. Then
by Corollary~\ref{cor:fF}, the MCF $f(\cdot, t)$ of $M$ in $\mathbb{H}^{m}(-1)$ exists for
all $t \in (-\infty, T)$, and it collapses to a focal submanifold of $M$ when $t$ goes to $T$.

Now we consider the limit of $f(\cdot, t)$ as $t\rightarrow-\infty$.
Note that $V$ is a Lorentz space and $V^\perp$ is a Euclidean space. We can
choose an orthonormal basis  $\{\alpha_1,\cdots, \alpha_{l+1}\}$ of $V$ and
 an orthonormal basis $\{\beta_{1},\cdots,\beta_{m-l}\}$ of $V^{\perp}$
such that
\[  \left<\alpha_{i},\alpha_{i}\right>=\left<\beta_{j},\beta_{j}\right> = 1, \hspace{10pt}
\left<\alpha_{l+1},\alpha_{l+1}\right>=-1, \hspace{10pt}
\left<\alpha_{l+1}, (0, \cdots, 0, 1) \right> < 0
\]
for $1 \leq i \leq l$ and $1 \leq j \leq m-l$.
Let $\epsilon = \{ \alpha_1, \cdots, \alpha_l, \beta_1, \cdots, \beta_{m-l}, \alpha_{l+1} \}$.
Then $\epsilon$ is an orthonormal basis of $\mathbb{R}^{m,1}$ satisfying condition
\eqref{eqn:condbasis}.
Let $\Psi_\epsilon : \mathbb{H}^{m}(-1) \longrightarrow \mathbb{B}^m$ be
the map $\Psi_{\epsilon, 1}$ defined by equation \eqref{eqn:PsiEr}.

For every
\[  x= \sum_{i=1}^{l+1} x_i \alpha_i \in V, \hspace{20pt}
   y=\sum_{j=1}^{m-l} y_j \beta_j \in V^\perp,\]
let
\[ \overline{x}:=(x_1, \cdots, x_{l}), \hspace{20pt} \overline{y}=(y_1, \cdots, y_{m-l}).\]
 Then
$\|\overline{x}\|^2  + \|\overline{y}\|^2 = x_{l+1}^2 -1$ if $(x, y) \in \mathbb{H}^{m}(-1)$ and
\begin{equation} \label{eqn:psiV}
\Psi_\epsilon (x, y) = \frac{1}{1+x_{l+1}} (\overline{x}, \,\, \overline{y}).
\end{equation}

By Corollary \ref{2.6} and equation (1.1) in \cite{X.C}, MCF of $\mathbb{H}^l(-r)$ in $V$
and MCF of $M'$ in $V^\perp$ can be written as
\[ F_1(x,t)=a_1(t) x, \hspace{10pt} F_2(y,t)= a_2(t) f_2(y,q(t)) \]
for $x \in \mathbb{H}^l(-r) \subset V$ and $y \in M' \subset V^\perp$,
 where
\[ a_1 (t) := \sqrt{1+\frac{2lt}{r}}, \hspace{10pt}
 a_2 (t) := \sqrt{1-\frac{2(n-l)}{r-1}t},  \hspace{10pt}
 q(t):=-\frac{r-1}{2(n-l)}\ln \left( 1-\frac{2(n-l)}{r-1}t \right),
\]
and $f_2(.,t)$ is the MCF of $M'$ in the sphere $S^{m-l-1}(r-1) \subset V^\perp$.
By equations (\ref{4.18}) and \eqref{7.1.6}, MCF of $M$ in $\mathbb{H}^m(-1)$ is given by
\begin{align*}
	f((x,y),t)=e^{-nt}(a_1(w(t)) x, \,\, a_2(w(t)) f_2(y,q(w(t)))) \in V \times V^\perp,
\end{align*}
where
\[ w(t) := \frac{1}{2n}(e^{2nt}-1).\]
Hence
\[ \Psi_\epsilon (f((x,y),t))
   = b(x, t) \left( \overline{x}, \,\, \frac{a_2(w(t))}{a_1(w(t))} \,\, \overline{f_2(y,q(w(t)))} \right) ,\]
where
\[ b(x, t):= \frac{e^{-nt} a_1(w(t))}{1+ e^{-nt} a_1(w(t)) x_{l+1}}. \]
Since
\[ \lim_{t \rightarrow -\infty} w(t) = -\frac{1}{2n}, \hspace{10pt}
\lim_{t \rightarrow -\infty} a_1(w(t)) = \sqrt{1 - \frac{l}{nr}} > 0, \hspace{10pt}
\lim_{t \rightarrow -\infty} a_2(w(t)) =  \sqrt{1+\frac{n-l}{n(r-1)}},
\]
we have
\[ \lim_{t \rightarrow -\infty} b(x, t) = \frac{1}{x_{l+1}}, \hspace{20pt}
\lim_{t \rightarrow -\infty} \frac{a_2(w(t))}{a_1(w(t))} = \sqrt{\frac{r}{r-1}}, \]
\[ \lim_{t \rightarrow -\infty} q(w(t)) = q* := -\frac{r-1}{2(n-l)}\ln \left( 1+\frac{n-l}{n(r-1)} \right) \in (-\infty, 0).
\]
Since $x \in \mathbb{H}^l(-r)$ and $f_2(y,q*) \in S^{m-l-1}(r-1)$, we have
\begin{equation} \label{eqn:limitPsif}
 \lim_{t \rightarrow -\infty}  \Psi_\epsilon (f((x,y),t))
= \frac{1}{x_{l+1}} \left( \overline{x}, \,\,  \sqrt{\frac{r}{r-1}} \,\, \overline{f_2(y,q*)} \right) \in S^{m-1}.
\end{equation}
Hence $f(\cdot, t)$ converges to a subset in the ideal boundary of $\mathbb{H}^{m}(-1)$.
Note that
$\sqrt{\frac{r}{r-1}} f_2(y,q*)$ lies in the sphere $S^{m-l-1}(r)$.

By Theorem 1.1 in \cite{X.C}, $f_2(\cdot, t)$ exists for all $t \in (-\infty, 0)$ and the image of $f_2(\cdot, t)$
is an isoparametric submanifold of $S^{m-l-1}(r-1)$ parallel to $M'$. In particular, the image of $f_2(\cdot, t)$
is diffeomorphic to $M'$. Consequently
\[ M'' := \left\{ \left. \sqrt{\frac{r}{r-1}} f_2(y,q*) \right| y \in M' \right\} \]
is an isoparametric submanifold of $S^{m-l-1}(r) \subset V^\perp$ which is diffeomorphic to $M'$.
Let
$\Phi:\mathbb H^l(-r)\times S^{m-l-1}(r)\rightarrow S^{m-1}$ be the map defined by
\begin{align*}
	\Phi(x,z) := \frac{1}{x_{l+1}} (\overline{x}, \overline{z}),
\end{align*}
for $x \in \mathbb H^l(-r) \subset V$ and $z \in S^{m-l-1}(r) \subset V^\perp$.
Then $\Phi$ is a smooth injective map. Let $\Phi_*$ be the differential map of $\Phi$.
For any
$ (v, u) \in T_{(x, z)} H^l(-r)\times S^{m-l-1}(r) $
with $v \in T_x H^l(-r) \subset V$ and $u \in T_z S^{m-l-1}(r) \subset V^\perp$, we have
\[ \Phi_*(v, u) = \frac{1}{x_{l+1}} (\overline{v}, \overline{u}) - \frac{v_{l+1}}{x_{l+1}^2} (\overline{x}, \overline{z}). \]
Equip $S^{m-1}$ with the standard metric induced from Euclidean metric on $\mathbb{R}^m$.
We have
\begin{eqnarray*}
\| \Phi_*(v, u) \|^2
&=& \frac{1}{x_{l+1}^2} (\| \overline{v} \|^2 + \|\overline{u}\|^2) + \frac{v_{l+1}^2}{x_{l+1}^4} (\|\overline{x}\|^2 + \|\overline{z}\|^2)
    - 2 \frac{v_{l+1}}{x_{l+1}^3} (\left< \overline{v}, \overline{x} \right> + \left<\overline{u}, \overline{z} \right> ).
\end{eqnarray*}
Since
\[ \|\overline{x}\|^2 + \|\overline{z}\|^2 = x_{l+1}^2, \hspace{10pt}
\left< \overline{v}, \overline{x} \right> - v_{l+1} x_{l+1} = 0, \hspace{10pt}
\left<\overline{u}, \overline{z} \right> = 0, \]
we have
\begin{eqnarray*}
\| \Phi_*(v, u) \|^2
&=& \frac{1}{x_{l+1}^2} (\|\overline{v}\|^2 - v_{l+1}^2 + \|\overline{u}\|^2) = \frac{1}{x_{l+1}^2} \|(v, u) \|^2.
\end{eqnarray*}
Hence $\Phi$ is a conformal diffeomorphism  from $\mathbb H^l(-r)\times S^{m-l-1}(r)$ to
an open subset of $S^{m-1}$.
Since the normal bundle of $\mathbb H^l(-r)\times M''$ in $\mathbb H^l(-r)\times S^{m-l-1}(r)$ is flat,
by Lemma \ref{flat}, the normal bundle of $\widetilde{M}:= \Phi(\mathbb H^l(-r)\times M'')$
in $S^{m-1}$  is also flat.
By equation \eqref{eqn:limitPsif}, $\Psi_\epsilon(f(\cdot, t))$ converges to
$\widetilde{M}$ which is diffeomorphic to $M$.
The proposition is thus proved.
$\Box$

If $M$ is not full in $\mathbb{H}^m(-1)$, then it is contained in a totally umbilical hypersurface
of the form $L(V, u) =  \mathbb{H}^m(-1) \bigcap (V+u)$ for some linear hyperplane $V \subset \mathbb{R}^{m,1}$
and $u \in \mathbb{R}^{m,1}$. For later applications, it is convenient to use another description
of $V+u$. Note that dimension of $V^\perp$ is $1$. We can choose
a non-zero vector $\xi \in V^\perp$ such that
\begin{equation} \label{eqn:condxi}
 a := \left<u, \, \xi \right> \geq 0, \hspace{40pt}  \langle\xi, \,\, \xi\rangle = \pm 1 \,\,\, {\rm or} \,\,\, 0.
\end{equation}
Then
\begin{align} \label{eqn:V+u}
	 		V+u=\{v\in\mathbb R^{m,1}\mid \left<v,\xi \right>=a\}.
\end{align}


\begin{rmk} \label{rem:LV-geo}
$L(V, u)$ is totally geodesic in $\mathbb{H}^m(-1)$ if and only if $a=0$. To see this, we first observe
 that $L(V, u)$ is totally geodesic in $\mathbb{H}^m(-1)$ if and only if
$V+u$ is a linear Lorentzian subspace of $\mathbb{R}^{m,1}$ (see, for example, \cite{BCO}).
If $V+u$ is linear, then $u \in V$ and $a=0$. On the other hand, if $a=0$ and $L(V, u)$ is non-empty,
then $V+u=V$ is linear,
and $\xi$ must be a spacelike vector by Theorem 4.9.1 in \cite{N.S}.
This implies that $V = (\mathbb{R} \xi)^\perp$ is a linear Lorentzian subspace, and
therefore $L(V, u)$ is totally geodesic in $\mathbb{H}^m(-1)$.
\end{rmk}

To prove Theorem \ref{1.17},  we need the following four lemmas.	
\begin{lem}\label{4.4}
	Assume $M$ is an $n$-dimensional submanifold of $L(V,u)  \subset \mathbb{H}^{m}(-1)$. Let $H_1(x)$ (respectively $H(x)$) be the mean curvature vector of $M$ in  $L(V,u) $ (respectively in $\mathbb{H}^{m}(-1)$) at $x\in M$. Then we have
	\begin{align*}
		H(x)= H_1(x) - n\alpha(\alpha x+\beta \xi),
	\end{align*}
	where
    \begin{equation} \label{4.12}
		\beta :=\frac{1}{\sqrt{\left<\xi, \, \xi \right>+a^2}}, \hspace{10pt}  \,\,\,\,\,\, \alpha :=\beta a,
    \end{equation}
    and $\xi$, $a$ are defined before equation \eqref{eqn:V+u}.
\end{lem}
\noindent
{\bf Proof}:
		The proof of this lemma is similar to the proof of Lemma \ref{4.1}. Since the orthogonal projection of
$H(x)$ to the tangent space $T_x L(V,u)$ is $H_1(x)$, we only need to prove that the orthogonal
projection of $H(x)$ to the normal space $\nu_x L(V, u)$ in $\mathbb{H}^{m}(-1)$ is
$- n\alpha(\alpha x+\beta \xi)$.

It is straightforward to check that $\left<\xi, \, \xi\right>+a^2 >0$ if $L(V,u)$ is non-empty.
To prove this, we can use Theorem 4.9.1 in \cite{N.S} if $a=0$ and Theorem 4.8.10 in \cite{N.S}
(i.e. the reversed Cauchy-Schwarz inequality)
if $a \neq 0$ and $\left< \xi, \xi \right> <0$.  Hence $\alpha$ and $\beta$ are well defined.
Moreover it is also known that
$\eta(x):=\alpha x+\beta \xi$ defines a unit normal vector field of $L(V,u)$ in $H^{m}(-1)$
(see, for example, equation (3.1) in \cite{ABD} and Example 3 in Chapter 2.1 of \cite{S.C}).
Let $A$ be the shape operator of $M$ considered as a submanifold in $H^{m}(-1)$,
$\nabla$ and $\nabla^L$ the Levi-Civita connections on $H^{m}(-1)$ and $\mathbb{R}^{m,1}$
respectively. Then $\nabla^L_v \eta = \alpha v$ for any tangent vector $v$ since $\xi$ is constant.
Choose an orthonormal basis $\{e_1,\cdots, e_n\}$ of $T_x M$. We have
		\begin{align*}
		- \left< H(x),\eta(x)\right>
			&= -\sum_{i=1}^n \left< A_{\eta(x)} e_i, e_i \right>
             =  \sum_{i=1}^n \left< \nabla_{e_i} \eta,  e_i \right>
             =  \sum_{i=1}^n \left< \nabla^L_{e_i} \eta,  e_i \right>
			 =  \sum_{i=1}^n \left< \alpha e_i, e_i \right>
			 =  n \alpha.
		\end{align*}
Since $\nu_x L(V, u) = \mathbb{R} \eta(x)$,  the orthogonal projection of $H(x)$ to $\nu_x L(V, u)$ is
\[ \left< H(x),\eta(x)\right> \eta(x) = - n\alpha(\alpha x+\beta \xi).\]
The lemma is thus proved.	
$\Box$
	
\begin{rmk} \label{neq}
	For later applications, we observe that $\pm \frac{\beta}{\alpha+1} \xi \notin \mathbb{H}^m(-1)$ if $L(V, u)$ is not empty.
Otherwise we would have $\left<\xi, \, \xi \right> = -1$ and $\beta=\alpha+1$ by equation \eqref{eqn:condxi}. On the other hand, as noted in the proof of Lemma \ref{4.4}, $a^2 > - \left<\xi, \, \xi \right> = 1$ if $L(V, u)$ is non-empty.  This implies
$\alpha > \beta$, which contradicts to $\beta=\alpha+1$.
\end{rmk}

\begin{lem}\label{4.2.1}
		Minimal isoparametric  submanifolds in $\mathbb{H}^{m}(-1)$ must be totally geodesic.
\end{lem}
\noindent
{\bf Proof}:
Let $M$ be an $n$-dimensional minimal isoparametric submanifold in $\mathbb{H}^{m}(-1)$.	
	
If $M$ is full, then by Theorem \ref{thm:Decomp},
there exists $r \geq 1$ and a Lorentzian subspace $V$ of $\mathbb R^{m,1}$ with dimension $l+1$, such that
\begin{align*}
			M=\mathbb{H}^l(-r) \times M',
\end{align*}
where $\mathbb{H}^l(-r)$ is one sheet of the hyperboloid
$\{y\in V \mid \left<y,y\right>=-r\} \subset V$ and
$M'$ is an isoparametric submanifold in the sphere $S^{m-l-1}(r-1)\subset V^{\perp} $.
By Lemma \ref{4.1}, the mean curvature vector of $\mathbb{H}^l(-r)$ in $V$ at $x \in \mathbb{H}^l(-r)$
is $\frac{l}{r} x$.
Let $H'(y)$ be the mean curvature vector of $M'$ in $V^{\perp}$ at $y \in M'$.
The the mean curvature vector of $M$ in $\mathbb{R}^{m,1} = V \times V^\perp$ at $(x, y)$ is given by
$(\frac{l}{r} x, H'(y))$.
By Lemma \ref{4.1}, the mean curvature vector of $M$ in $\mathbb{H}^l(-1)$ is
		\begin{align*}
			H(x,y)&=(\frac{l}{r} x, H'(y))-n(x,y)
			      =((\frac{l}{r}-n)x, H'(y)-ny).
		\end{align*}
Since $l \leq n$ and $r \geq 1$, we have $\frac{l}{r}-n \leq 0$ with equality if and only $l=n$ and $r=1$.
Hence if $H(x,y)=0$, we must have $l=n$ and $r=1$, which in turn implies that
$M'$ is a point and $M = \mathbb{H}^l(-1)$ is a totally geodesic submanifold of $\mathbb{H}^m(-1)$.

We then prove the lemma by induction on the codimension  $k=m-n$ of $M$.
If $k=0$, then $M$ is totally geodesic in $\mathbb{H}^m(-1)$ since they are the same space.

Assume the lemma holds for minimal isoparametric submanifolds with codimension less than $k$.
We want to show any minimal isoparametric submanifold $M$ in $\mathbb{H}^m(-1)$ with codimension $k$
must be totally geodesic.
We have proved that $M$ is indeed totally geodesic if it is full. So we can assume that $M$ is not full.
In this case $M$ must be an isoparametric submanifold in some totally umbilical hypersurface
$L(V, u) = \mathbb{H}^m(-1) \bigcap (V+u)$ with $V+u$ given by equation \eqref{eqn:V+u}.

For $x\in M$, we denote $H_1(x)$ and $ H(x)$ the mean curvature vectors of $M$ in $L(V,u)$ and $\mathbb{H}^m(-1)$
respectively.
Then by Lemma \ref{4.4},
\[ H(x)=H_1(x) - n \alpha (\alpha x + \beta \xi),\]
 where $H_1(x) \in T_x L(V,u)$
and  $\alpha x + \beta \xi$ is a unit normal vector of $L(V,u)$ at $x$. By assumption, $H(x)=0$
since $M$ in minimal in $\mathbb{H}^m(-1)$. Hence we must have
$H_1(x)=0$ and $\alpha =0$ (we can assume $n>0$ here since otherwise $M$ is a point, which is automatically totally
geodesic). Since $\beta \neq 0$, $\alpha =0$ implies $a=0$ by equation \eqref{4.12}. Hence by Remark \ref{rem:LV-geo}, $L(V,u)$ is totally geodesic in $\mathbb{H}^m(-1)$ and is isometric to a lower dimensional hyperbolic space.
Since $H_1(x)=0$ for all $x \in M$, $M$ is a minimal isoparametric submanifold in the hyperbolic space $L(V,u)$ with codimension $k-1$. By the induction hypothesis, $M$ is totally geodesic in $L(V,u)$, and hence is also
totally geodesic in $\mathbb{H}^m(-1)$.
This completes the proof of the lemma.
$\Box$

\begin{lem} \label{lem:varphi}
For any $y=(y_1, \cdots, y_{m+1}) \in \mathbb{R}^{m,1}$, let $\overline{y}=(y_1, \cdots, y_m)$.
Define a map $\varphi: L(V, u) \longrightarrow S^{m-1}$  by
\[ \varphi(y) :=  \frac{\overline{y}+c \, \overline{\xi}}{y_{m+1}+c \, \xi_{m+1}}, \]
where $c := \frac{\beta}{\alpha + 1}$ and $(\overline{\xi}, \xi_{m+1})=\xi$. Then $\varphi$ is a conformal map.
\end{lem}
\noindent

{\bf Proof}:
We first check $\varphi$ is well defined and $\varphi(y) \in S^{m-1}$ if $y \in L(V, u)$. If fact, $y \in L(V, u)$ implies that
\[ \|\overline{y}\|^2=y_{m+1}^2-1, \hspace{10pt}
\langle\overline{y}, \overline{\xi}\rangle = a + y_{m+1} \xi_{m+1}, \hspace{10pt}
\|\overline{\xi}\|^2= \xi_{m+1}^2 + \langle\xi, \xi\rangle. \]
Hence
\[ \|\overline{y}+c \, \overline{\xi} \|^2 = (y_{m+1}+ c \xi_{m+1})^2 -1 + 2 ca  + c^2 \langle\xi, \xi\rangle.\]
Using definition of $\alpha, \beta, c$, it is straightforward to check $-1 + 2 ca  + c^2 \langle\xi, \xi\rangle=0$.
So we have
\begin{equation} \label{eqn:phinorm}
 \|\overline{y}+c \, \overline{\xi} \|^2 = (y_{m+1}+ c \xi_{m+1})^2.
\end{equation}
This implies $y_{m+1}+ c \xi_{m+1}\neq 0$ since otherwise we would have $y= -c \xi$, which is not possible by Remark \ref{neq}. Hence $\varphi$ is well defined. Moreover, equation \eqref{eqn:phinorm} also implies $\| \varphi(y) \| = 1$
and $\varphi(y) \in S^{m-1}$.

For any $w \in T_y L(V, u) \subset \mathbb{R}^{m,1}$, we have
\[ \varphi_*(w) = \frac{\overline{w}}{y_{m+1}+c\xi_{m+1}} -
     \frac{w_{m+1}(\overline{y}+c\overline{\xi})}{(y_{m+1}+c\xi_{m+1})^2}.\]
Since $\langle w, y\rangle=\langle w, \xi\rangle=0$, we have
$\langle\overline{w}, \overline{y}\rangle = w_{m+1}y_{m+1}$ and $\langle\overline{w}, \overline{\xi}\rangle = w_{m+1} \xi_{m+1}$.
Together with equation \eqref{eqn:phinorm}, we obtain
\[ \|\varphi_*(w)\|^2 =  \frac{\langle w, w\rangle}{(y_{m+1}+c\xi_{m+1})^2}.\]
Hence $\varphi$ is a conformal map.
$\Box$

\begin{lem} \label{F'formula}
Assume $M$  is an $n$-dimensional  submanifold in a totally umbilical submanifold $L(V, u) \subset \mathbb H^{m}(-1)$.
Let $f_1(\cdot,t)$ and $F(\cdot,t)$ be the solutions to  mean curvature flows for $M$ in $L(V,u)$ and  $\mathbb R^{m,1}$ respectively. Let $\xi$, $\alpha$ and $\beta$  be defined by equations \eqref{eqn:condxi} and \eqref{4.12} .
If $\alpha\neq 1$, then for $x\in M$,
	\begin{align}
		F(x,t)=\sqrt{2nt(1-\alpha^2)+1}\, \, (f_1(x,s_{\alpha}(t))-\eta)+\eta,  \label{9.1.5}
	\end{align}
	where $\eta := \frac{\alpha\beta}{1-\alpha^2}\xi$ and  $s_{\alpha}(t)$ is defined by
	\begin{align} \label{eqn:st}
		 	s_{\alpha}(t) :=
		 	\frac{1}{2n(1-\alpha^2)}\ln (2nt(1-\alpha^2)+1).
     \end{align}
If $\alpha=1$, then
\begin{align}
		  	 F(x,t)&=f_1(x,t)-  nt  \beta \xi .  \label{9.1.6}
		  \end{align}
\end{lem}

\begin{proof}
If $\alpha \neq 1$, it is straightforward to check that $\eta \in V + u$ and $L(V, u)$ consists of points $x \in V+u$ satisfying equation
$ \langle x-\eta, x-\eta\rangle = (\alpha^2-1)^{-1}. $
Let $\tau$ be the translation on $\mathbb R^{m,1}$ defined by $\tau(x)=x-\eta$ for $x\in \mathbb R^{m,1}$.
Then $\tau(L(V, u))$ is a sphere (or hyperbolic space)  \footnote{$\tau(L(V, u))$ is a circle or hyperbola if its dimension is one.}
centered at origin in the Eulidean (or Lorentz) space
$\tau(V+u)$ if $\alpha > 1$ (or $0 \leq \alpha < 1$).
Since $\tau$ is an isometry, $\tau(f_{1}(\cdot,t))=f_{1}(\cdot,t)-\eta$ is the solution to MCF for $\tau(M)$ in $\tau(L(V,u))$. By Lemma \ref{1.2}  and its spherical analogue (cf. Equation (1.1) in \cite{X.C}), the function
\[ G(x, t) := \sqrt{2nt(1-\alpha^2)+1}\, \, (f_1(x,s_{\alpha}(t))-\eta)\]
 is the solution to MCF for $\tau(M)$ in $\tau(V+u)$. Since $\tau(V+u)$ is a totally geodesic submanifold in $\mathbb R^{m,1}$. $G(x, t)$ also solves the MCF of $\tau(M)$ in $\mathbb R^{m,1}$.
 Hence the solution of MCF for $M$ in $\mathbb R^{m,1}$ is $\tau^{-1}(G(x,t))$, which is exactly the
 function given by equation \eqref{9.1.5}.

If $\alpha=1$, then $\langle \xi, \xi \rangle =0$ and the induced metric on $\tau(V+u)$ is degenerate. In this case, MCF for submanifolds in $\tau(V+u)$ is not well defined. So the above arguments do not work for this case.
Instead, we will check directly that the function $F(x, t)$ in equation \eqref{9.1.6}, which is the limit of the function in
equation \eqref{9.1.5} as $\alpha \rightarrow 1$, satisfy the MCF equation.
Let $H_1(\cdot,t)$ be the mean curvature vector field of $f_1(\cdot,t)$ in $L(V,u)$. Then
\begin{align}
\frac{\partial}{\partial t}F(x,t)
&=H_{1}(x,t)-n\beta\xi.\label{2.9}
\end{align}
On the other hand, since $f_1(\cdot ,t)\subset L(V,u) \subset \mathbb{H}^{m}(-1)$,
by Lemma \ref{4.4}, the mean curvature vector of $f_1(\cdot ,t)$ in $\mathbb{H}^{m}(-1)$ at point $f_1(x,t)$ is
		 \begin{align*}
		 	H_1(x,t) -n(f_1(x,t)+\beta \xi).
		 \end{align*}
By Lemma \ref{4.1},
the mean curvature vector of $f_1(\cdot ,t)$ in $\mathbb R^{m,1}$ at point $f_1(x,t)$ is
$ H_1(x,t)- n \beta \xi$,
which is exactly the right hand side of equation \eqref{2.9}.
Since for each fixed $t$, $F(\cdot, t)$ is just a translation of $f_{1}(\cdot, t)$,
the right hand side of equation~\eqref{2.9} is also the mean curvature vector of
$F(\cdot,t)$  in $\mathbb R^{m,1}$ at the point $F(x,t)$.
Hence $F(\cdot, t)$ is the solution to MCF of $M$ in $\mathbb R^{m,1}$.
This completes the proof of the lemma.
\end{proof}

Now we are ready to prove Theorem \ref{1.17}.

\noindent
{\bf Proof of Theorem \ref{1.17}}:
Let $M$ be an $n$-dimensional isoparametric submanifold of $\mathbb{H}^{m}(-1)$. We prove this
theorem by induction on the codimension $k=m-n$ of $M$.
When $k=0$, $M=\mathbb{H}^{m}(-1)$ and the theorem holds trivially.
For $k \geq 1$, assuming the theorem holds for all isoparametric submanifolds in hyperbolic spaces
with codimension less than $k$, we want to prove the theorem also holds for $M$ with codimension $k$.

If $M$ is full in $\mathbb{H}^{m}(-1)$, the theorem follows from Proposition \ref{1.9}.
So we only need to consider the case where $M$ is not full in $\mathbb{H}^{m}(-1)$.
In this case,  $M$ is contained in a totally umbilical submanifold $L(V, u) \subset\mathbb H^{m}(-1) $.
Let $f_1(\cdot,t)$  be the solution to  MCF for $M$ in $L(V,u)$. Then the MCF $F(\cdot, t)$ of $M$ in $\mathbb R^{m,1}$
is given by Lemma~\ref{F'formula}.
In order to construct the MCF $f(\cdot,t)$  of $M$ in $\mathbb{H}^{m}(-1)$ from $F(\cdot, t)$, we need to make sure the image of
$F(\cdot ,t)$ lies in the time cone $\mathcal{TC}$ defined by equation \eqref{eqn:timecone}.
Using equations \eqref{9.1.5} and \eqref{9.1.6}, it is straightforward to check that
\begin{align*}
\left<F(x,t), F(x,t)\right>
		 	&=-1-2nt, \hspace{10pt} {\rm for \,\,\, all \,\,\,} x \in M.
\end{align*}
Hence for all $t > -\frac{1}{2n}$, the image of $F(\cdot,t)$ lies in $\mathcal{TC}$ and
we can construct $f(\cdot,\cdot)$ from $F(\cdot, t)$ using equation~\eqref{4.18}.
We then complete the proof of this theorem in four steps.

{\bf Step 1}: Prove $f(\cdot, t)$ is an ancient solution.

Observe that $L(V,u)$ is either a Euclidean space, a sphere,
or a hyperbolic space with dimension less than $m$.
By Theorem 1.1  in \cite{X.C}, Theorem 1.1 in \cite{X.CT}, and the induction hypothesis, the maximal existence interval for $f_1(\cdot, t)$ is $(-\infty, T')$ for some $T' \geq 0$. Moreover,
$f_1(\cdot, t)$ collapses to a focal submanifold of $M$ if $T' < \infty$.
Let $T''=T'$ if $\alpha=1$ and
\begin{equation} \label{eqn:T''}
 T'' =\frac{e^{2n(1-\alpha^2)T'}-1}{2n(1-\alpha^2)}  {\rm \,\,\, if \,\,\, } \alpha \neq 1.
\end{equation}
Then $T''\geq 0$ for all $\alpha\geq 0$ and
it is straightforward to check $s_\alpha(t) < T'$ if and only if $t < T''$, where $s_{\alpha}(t)$ is defined by
equation \eqref{eqn:st} if $\alpha \neq 1$ and $s_1(t) := t$. By equation \eqref{9.1.5}, for $F(\cdot, t)$ to be well defined,
we also need $2nt(1-\alpha^2)+1 \geq 0$, which is automatically true if $t < T''$ and $\alpha \geq 1$ and requires
$t > - \frac{1}{2n(1-\alpha^2)}$ if $0 \leq \alpha < 1$.
Since $- \frac{1}{2n(1-\alpha^2)} < - \frac{1}{2n}$ when $0 \leq \alpha < 1$,
we know $F(\cdot, t)$ exists at least over the interval $(- \frac{1}{2n}, T'')$ for all $\alpha \geq 0$.

Let
\begin{equation} \label{eqn:T}
 T = \frac{1}{2n} \ln(1+2nT'') \geq  0.
\end{equation}
By Corollary \ref{cor:fF}, $f(\cdot, t)$ given by equation \eqref{4.18} is well defined for all $t \in (-\infty, T)$
and gives an ancient solution to MCF of $M$ in $\mathbb{H}^m(-1)$.
So it only remains to study the limiting behaviours of $f(\cdot, t)$.

\vspace{6pt}
\textbf{Step 2}: Prove Parts (1)--(4) of Theorem \ref{1.17}. The proof is divided into the following three cases according  to the type of $L(V,u)$.

\textbf{Case 1}: Study limits of $f(\cdot, t)$ when $L(V,u)$ is a hyperbolic space.
In this case, $0\leq \alpha<1$ and $\langle\xi,\xi\rangle=1$.

{\it Part (1) of Theorem \ref{1.17}}: If $T<\infty$, then $T'<\infty$. By the induction hypothesis, $f_1(\cdot, t)$ collapses to a focal submanifold
as $t$ goes to  $T'$. This implies that $F(\cdot, t)$
collapses to a focal submanifold of $M$ as $t \rightarrow T''$. Consequently
$f(\cdot, t)$ collapses to
a focal submanifold of $M$ as $t \rightarrow T$.

{\it Part (2) of Theorem \ref{1.17}}: If $T=\infty$, then $T'=T''=\infty$. If in addition $M$ is not flat,  
then by the induction hypothesis, $f_1(\cdot,t)$ converges to a totally geodesic
submanifold $N \subset L(V,u) $  as $t \rightarrow \infty$. In particular, this means for all $x \in M$,
\begin{align}
     	\lim_{t\rightarrow \infty} f_1(x,t) = h(x),     \label{4.20}
\end{align}
where $h: M\rightarrow L(V,u)$ is the embedding of $N$ in $L(V,u)$.
		
By  equations (\ref{4.18}) and (\ref{9.1.5}), we have
		  \begin{align}
		  	f(x,t) &= v_\alpha(t) f_1(x,s_{\alpha}(w(t)))
		  	      -(v_\alpha(t)-e^{-nt}) \, \eta,   \label{4.21}
		  \end{align}
where $\eta$ is the constant vector defined in Lemma \ref{F'formula} and
\begin{equation} \label{eqn:wv}
w(t) :=\frac{1}{2n}(e^{2nt}-1), \hspace{10pt}
v_\alpha(t) := \sqrt{(1-e^{-2nt})(1-\alpha^2)+e^{-2nt}}.
\end{equation}		
As $t \to \infty$, we have
\begin{equation} \label{eqn:limv}
 v_\alpha(t) \to \sqrt{1-\alpha^2}, \,\,\,
 w(t) \to \infty, \,\,\,
 s_{\alpha}(w(t))	\to \infty, \,\,\,
 f_1(x,s_{\alpha}(w(t))) \to h(x).
\end{equation}
Consequently, for all $x \in M$, we have
\begin{equation} \label{eqn:limf}
 \lim_{t \to \infty} f(x,t) = \tilde{h}(x),
\end{equation}
where $\tilde{h}: M \rightarrow \mathbb R^{m,1}$ is the map defined by
\begin{align} \label{eqn:htilde}
		  	\tilde{h}(x)=\sqrt{1-\alpha^2} \, (h(x)-\eta).
\end{align}
By construction, $f(x,t) \in \mathbb{H}^m(-1)$ for all $x \in M$ and $t \in \mathbb{R}$.
Hence $\tilde{h}(x) \in  \mathbb{H}^m(-1)$ for all $x \in M$.
Let $\tilde{N} = \tilde{h}(M) \subset \mathbb{H}^m(-1)$. Then $f(\cdot, t)$ pointwise converges to $\tilde{N}$.

We now show that $\tilde{N}$ is a totally geodesic submanifold  in $\mathbb{H}^m(-1)$.
In fact, since $N$ is totally geodesic in $L(V, u)$, the mean curvature vector of $N$ in $L(V,u)$ is $0$.
Applying Lemma~\ref{4.1} and Lemma \ref{4.4} to the inclusion
$N \subset L(V,u)\subset \mathbb{H}^m(-1)$,
we can show that for all $x \in M$, the mean curvature vector of $N$ in $\mathbb{R}^{m,1}$ at point $h(x)$ is equal to
\begin{align*}
	-n\alpha(\alpha h(x)+\beta \xi)+nh(x) \,\, = \,\, \sqrt{1-\alpha^2} \, n \tilde{h}(x).
\end{align*}
Observe that $\tilde{N}$ is just a translation of a  multiplication of $N$ by a scalar $\sqrt{1-\alpha^2}$.
So the mean curvature vector of $\tilde N$ in $\mathbb{R}^{m,1}$ at point $\tilde h(x)$ is equal to $n \tilde{h}(x)$.
By Lemma \ref{4.1}, the mean curvature vector of $\tilde N$ in $\mathbb{H}^m(-1)$ at point $\tilde h(x)$
is equal to $0$. Hence $\tilde{N}$ is minimal in $\mathbb{H}^m(-1)$.

Moreover, since $N$ is totally geodesic in $L(V,u)$, by Proposition 1.3 in \cite{B.W} and Lemma \ref{1.15}, $N$ is isoparametric in  both $\mathbb{H}^m(-1)$ and $\mathbb R^{m,1}$.
Since $\tilde{N}$ is just a translation of a scalar multiplication of $N$,
$\tilde{N}$ is  also isoparametric in  both $\mathbb{H}^m(-1)$ and $\mathbb R^{m,1}$.
Since $\tilde{N}$ is also minimal in $\mathbb{H}^m(-1)$, by Lemma \ref{4.2.1},
$\tilde{N}$ is totally geodesic in $\mathbb{H}^m(-1)$.

{\it Part (3) of Theorem \ref{1.17}}: If $T=\infty$  and $M$ is flat with dimension greater than one, then $T'=\infty$ and by induction hypothesis, $f_1(\cdot, t)$
converges to a point on the ideal boundary of $L(V,u)$ as $t \rightarrow \infty$.
Since $V$ is a Lorentzian subspace of $\mathbb{R}^{m,1}$,
we can choose an orthonormal basis $\epsilon=\{\epsilon_1, \cdots, \epsilon_{m+1}\}$ of $\mathbb{R}^{m,1}$
satisfying condition \eqref{eqn:condbasis} such that $\{ \epsilon_2, \cdots, \epsilon_{m+1} \}$ is an orthonormal
basis of $V$. 
By Lemma \ref{lem:convBD}, as $t \rightarrow \infty$, $ \langle f_1(x,t), \, \epsilon_{m+1}\rangle \rightarrow -\infty$ and
$\frac{f_1(x, t)}{-\langle f_1(x,t), \, \epsilon_{m+1}\rangle}$ converges to a point in $V+u$
which is independent of $x$.
Since as $t \rightarrow \infty$,
$v_\alpha(t) \rightarrow \sqrt{1-\alpha^2}$ and
$s_{\alpha}(w(t)) \rightarrow \infty$.
So by equation \eqref{4.21},
$ \langle f(x,t), \epsilon_{m+1}\rangle \rightarrow -\infty$ and
\[ \lim_{t \rightarrow \infty} \,\, \frac{f(x, t)}{-\langle f(x,t), \epsilon_{m+1}\rangle}
  \,\, = \,\, \lim_{t \rightarrow \infty}  \,\, \frac{f_1(x, t)}{-\langle f_1(x,t), \epsilon_{m+1}\rangle} \]
  converges to a point in $\mathbb{R}^{m,1}$ independent of $x$. 
Hence, by Lemma \ref{lem:convBD}, $f(\cdot, t)$ converges to a point on the ideal boundary of $\mathbb{H}^m(-1)$.

{\it Part (4) of Theorem \ref{1.17}}:
Assume $M$ is not totally geodesic in $\mathbb{H}^m(-1)$ and consider the limit of $f(\cdot, t)$ as $t \rightarrow -\infty$.

If $\alpha=0$, then $v_{\alpha}(t)=1$,  $s_{\alpha}(w(t))=t$, and $f(x, t) = f_1(x, t)$. By Remark \ref{rem:LV-geo},
$L(V, u)$ is totally geodesic in $\mathbb{H}^m(-1)$ and $V+u=V$ is a Lorentzian subspace of $\mathbb{R}^{m,1}$.
We can choose
an orthonormal basis $\epsilon=\{\epsilon_1, \cdots, \epsilon_{m+1}\}$ of $\mathbb{R}^{m,1}$ satisfying
condition \eqref{eqn:condbasis} such that
$\epsilon_1 = \xi$ and $\{\epsilon_2, \cdots, \epsilon_{m+1}\}$ is an orthonormal basis of $V$.
Let $\Psi_{\epsilon,1}: \overline{\mathbb{H}^m(-1)} \longrightarrow \overline{\mathbb{B}^m}$ be the map defined by
equation \eqref{eqn:PsiEr}. Then $\Psi_{\epsilon,1}$ maps $L(V,u)$ to the set
$\{(0, a_2, \cdots, a_m) \mid \sum_{i=2}^m a_{i}^2 < 1\}$. In particular, it maps
the ideal boundary of $L(V,u)$ to a totally geodesic subsphere in $S^{m-1}$.
By induction hypothesis, $f_1(\cdot, t)$ converges to a smooth submanifold with flat normal bundle
in the ideal boundary of $L(V,u)$ as $t \rightarrow -\infty$.
Hence $f(\cdot, t)$ also converges to a smooth submanifold with flat normal bundle
in the ideal boundary of $\mathbb{H}^m(-1)$. Moreover, the dimension of these limit submanifolds are the same
as the dimension of $M$.

If $0<\alpha<1 $, we define
\begin{align*}
	T_\alpha:= \lim_{t \rightarrow -\infty} s_\alpha(w(t)) =
		\frac{1}{2n(1-\alpha^2)}\ln\alpha^2.
\end{align*}
Then $T_\alpha <0$.
Let
\[ z(x, t) := e^{nt} v_\alpha(t) f_1(x, s_{\alpha}(w(t)))
		  	      +(1-e^{nt} v_\alpha(t)) \,\, \eta.\]
Then $f(x, t) = e^{-nt} z(x, t)$.

Note that, as $t \rightarrow -\infty$,  $e^{nt} v_\alpha(t) \rightarrow \alpha$  and
\[ z(x, t) \rightarrow \alpha \left( f_1(x,T_\alpha)
		  	      + c \,\, \xi \right), \]
where $c:= \frac{\beta}{\alpha+1}$.
Since $f(x, t) \in \mathbb{H}^{m}(-1)$, the last coordinate of $z(x, t)$, denoted by $z_{m+1}(x, t)$, must be positive. Together with Remark \ref{neq}, we have
\[ \lim_{t \rightarrow -\infty} z_{m+1}(x, t) =
      \alpha \left( f_1(x,T_\alpha)_{m+1} +c \, \xi_{m+1} \right) \in (0, \infty), \]
 where
$f_1(x,T_\alpha)_{m+1}$ is the last coordinate of $f_1(x,T_\alpha)$. Hence
the last coordinate of $f(x, t)$ goes to $\infty$ as  $t \rightarrow -\infty$.
Let $\psi$ be the map from $\mathbb{H}^{m}(-1)$ to $\mathbb{B}^m$ defined by equation \eqref{9.2}.
By Lemma \ref{lem:convBD},
\[ \lim_{t \rightarrow -\infty} \psi(f(x,t))
 =  \lim_{t \rightarrow -\infty} \frac{\overline{z(x,t)}}{z_{m+1}(x, t)}
    =  \frac{\overline{f_1(x,T_\alpha)} + c \overline{\xi}}{f_1(x, T_\alpha)_{m+1} + c \xi_{m+1}} \in S^{m-1}, \]
where $\overline{z(x,t)}$ and $\overline{f_1(x,T_\alpha)}$ are vectors whose coordinates are
the first $m$ coordinates of $z(x, t)$ and $f_1(x, T_\alpha)$ respectively.
Hence $\psi(f(\cdot, t))$ converges to $\widetilde{M} := \{\varphi(f_1(x, T_\alpha)) \mid x \in M \}$
where $\varphi: L(V, u) \longrightarrow S^{m-1}$ is a conformal diffeomorphism from
$L(V, u)$ to an open submanifold of $S^{m-1}$ defined in Lemma \ref{lem:varphi}.
Note that $f_1(\cdot, T_\alpha)$ is an isoparametric submanifold in $L(V,u)$.
By Lemma \ref{flat}, $\widetilde{M}$ is a smooth submanifold with flat normal bundle
in the ideal boundary of $\mathbb{H}^m(-1)$. Moreover dimension of $\widetilde{M}$
is equal to dimension of $f_1(\cdot, T_\alpha)$, which is also equal to dimension of $M$.

\vspace{6pt}

\textbf{Case 2}: Study limits of $f(\cdot, t)$ when $L(V,u)$ is a sphere.
In this case,   $ \alpha>1$ and $\langle\xi,\xi\rangle=-1$.

We first consider part (1) of Theorem \ref{1.17}. If $T < \infty$, then $T'$ could be either finite or infinite.
If $T' < \infty$, then the proof is the same as the corresponding part in Case 1.  If $T'=\infty$, by Theorem 1.1 in \cite{X.CT}, $M$ must be minimal in $L(V,u)$,
which also implies that $f_1(\cdot,t)$ is independent of $t$.
Note that  $T''= \frac{1}{2n(\alpha^2-1)}$ if $T'=\infty$.
Since $\lim_{t \rightarrow T''} \,\, 2nt(1-\alpha^2)+1 =0$, $F(\cdot, t)$ converges to a point when $t\rightarrow T''$.
Consequently  $f(\cdot, t)$ also converges to a point (which is also a focal submanifold of $M$) when $t\rightarrow T$.

Parts (2) and (3) of Theorem \ref{1.17} can not occur since $T$ is always finite in this case.
The proof of part (4) of Theorem \ref{1.17} is the same as the corresponding part for the case of $0<\alpha<1$ in Case 1.

\vspace{6pt}

\textbf{Case 3}: Study limits of $f(\cdot, t)$ when $L(V,u)$ is isometric to a Euclidean space.
 In this case, $\alpha=1$ and $\left<\xi, \xi \right>=0$.

 The proof of part (1) of Theorem \ref{1.17}  is the same as the corresponding part in Case 1.
 The proof of part (4)  of Theorem \ref{1.17} can be obtained from corresponding part for the case of $0 < \alpha < 1$ in Case 1 by setting $T_1 = - \frac{1}{2n}$ and taking limits for all other formulas as $\alpha \rightarrow 1$.

If $T = \infty$, then $T''=T'=\infty$. By Theorem 1.1 in \cite{X.CT},  $T'=\infty$ if and only if $M$ is totally geodesic in $L(V,u)$. In particular, $M$ is flat.  Hence part (2)  of Theorem \ref{1.17} can not  occur and we only need consider part (3) of  Theorem \ref{1.17} in this case. 
Since $M$ is totally geodesic in $L(V,u)$, $f_1(x,t)=x$ for all $x \in M$ and $t$. By equations \eqref{4.18} and \eqref{9.1.6}, we have
\[ f(x, t) = e^{-nt} x + \frac{\beta}{2} (e^{-nt}-e^{nt}) \xi.\]
For $x=(x_1, \cdots x_{m+1}) \in \mathbb{R}^{m,1}$, we write
$x=(\overline{x}, x_{m+1})$ where $\overline{x}=(x_1, \cdots, x_m)$.
Let $\psi$ be the map from $\mathbb{H}^{m}(-1)$ to $\mathbb{B}^m$ defined by equation \eqref{9.2}.
Then
\[ \psi(f(x, t)) = \frac{ e^{-nt} \overline{x} + \frac{\beta}{2} (e^{-nt}-e^{nt}) \overline{\xi}}
                          {1+ e^{-nt} x_{m+1} + \frac{\beta}{2} (e^{-nt}-e^{nt}) \xi_{m+1}}. \]
Hence for all $x \in M$,
\[ \lim_{t \rightarrow \infty} \psi(f(x, t)) = \frac{\overline{\xi}}{\xi_{m+1}}.\]
Since $\langle\xi, \xi\rangle= \|\overline{\xi}\|^2 - \xi_{m+1}^2 =0$, $\, \, \frac{\overline{\xi}}{\xi_{m+1}} \in S^{m-1}$.
This shows that $f(\cdot, t)$ converges to a point in the ideal boundary of  $\mathbb{H}^{m}(-1)$.

This completes the proof of Theorem \ref{1.17}.
$\Box$

\begin{rmk} \label{dimension,one}

If dimension of $M$ is one, then sectional curvature of $M$ is always 0 (i.e. $M$ is always flat). In this case, the statement of part (3) of Theorem \ref{1.17} is no longer accurate and it should be modified properly. In fact, by Theorem 1.11 in \cite{B.W}, a one-dimensional isoparametric submanifold $M$ in $\mathbb{H}^m(-1)$ must have one of the following 4 types: (i) a circle, (ii) a hyperbola inside the light cone of a translation of a 2-dimensional Lorentz subspace of $\mathbb{R}^{m, 1}$, (iii) a horocycle in a 2-dimensional hyperbolic space contained in $\mathbb{H}^m(-1)$, (iv) a geodesic in a flat totally umbilical hypersurface of dimension bigger than one in a hyperbolic space contained in $\mathbb{H}^m(-1)$. Assume $T=\infty$. Then $M$ can not be a circle.
If $M$ is a hyperbola as in type (ii), then MCF $f(\cdot, t)$ converges to a geodesic in $\mathbb{H}^m(-1)$  as $t \rightarrow \infty$, and the proof is the same as the proof of part (2) in Step 2 Case 1.
If $M$ is a horocycle as in type (iii), then  $f(\cdot, t)$ converges to a point in the ideal boundary of $\mathbb{H}^m(-1)$,
and the proof is the same as the proof of part (3) in Step 2 Case 1.
If $M$ has type (iv),  then $f(\cdot, t)$ also converges to a point in the ideal boundary of $\mathbb{H}^m(-1)$,
and the proof is the same as the proof of part (3) in Step 2 Case 3.
\end{rmk}
	
\begin{rmk} \label{rem:RecCons}
Since MCF flow of codimension $0$ submanifold is trivial, we can use equations \eqref{9.1.5} and \eqref{9.1.6}, Corollary \ref{cor:fF},
and constructions of solutions in spheres and Euclidean spaces in \cite{X.CT}
to obtain explicit solutions for MCF of all isoparametric submanifolds in hyperbolic spaces in a recursive manner.
\end{rmk}

\vspace{30pt} \noindent
Xiaobo Liu \\
School of Mathematical Sciences \& \\
Beijing International Center for Mathematical Research, \\
Peking University, Beijing, China. \\
Email: {\it xbliu@math.pku.edu.cn}
\ \\ \ \\
Wanxu Yang \\
School of Mathematical Sciences, \\
Peking University, Beijing, China. \\
Email: {\it yangwanxu@stu.pku.edu.cn}
	
\end{document}